\documentclass[12pt,amscd]{amsart}
\usepackage[all]{xy}
\usepackage{graphicx,tikz-cd}
\usepackage{amsmath,amsxtra,amssymb,latexsym, amscd,amsthm}
\usepackage{indentfirst}
\usepackage[mathscr]{eucal}
\usepackage[pagebackref=true]{hyperref}

\newtheorem{thm}{Theorem}[section]
\newtheorem{cor}[thm]{Corollary}
\newtheorem{lem}[thm]{Lemma}

\theoremstyle{definition}
\newtheorem{defn}[thm]{Definition}
\newtheorem{exm}[thm]{Example}

\newtheorem{rem}[thm]{Remark}

\numberwithin{equation}{section}

\DeclareMathOperator{\ZZ}{\mathbb {Z}}

\DeclareMathOperator{\Tor}{Tor}
\DeclareMathOperator{\ann}{ann}

\DeclareMathOperator{\pd}{pd}
\DeclareMathOperator{\supp}{supp}
\DeclareMathOperator{\type}{type}

\DeclareMathOperator{\reg}{reg}

\def\P {\mathbf P}

\def\s {\mathbf s}

\def\m {\mathfrak m}

\def\k {\mathrm{k}}

\begin{document}

\title{Cochordal zero divisor graphs and Betti numbers of their edge ideals}

\author[L. X. Dung]{Le Xuan Dung}
\address{Department of Algebra and Geometry, Hong Duc University, No. 565 Quang Trung Street, Dong Ve Ward, Thanh Hoa, Vietnam}
\email{lexuandung@hdu.edu.vn}

\author{Thanh Vu}
\address{Institute of Mathematics, VAST, 18 Hoang Quoc Viet, Hanoi, Vietnam}
\email{vuqthanh@gmail.com}


\date{}

\commby{}

\begin{abstract}
    We associate a sequence of positive integers, termed the type sequence, with a cochordal graph. Using this type sequence, we compute all graded Betti numbers of its edge ideal. We then classify all positive integer $n$ such that the zero divisor graph of $\ZZ/n \ZZ$ is cochordal and determine all the graded Betti numbers of its edge ideal.
\end{abstract}

\maketitle

\section{Introduction}
\label{sect_intro}

Rather, Imran, and Pirzada \cite{RIP} recently provided formulae for the Betti numbers of the edge ideals of zero divisor graphs of $\ZZ/n\ZZ$, denoted by $\Gamma(\ZZ_n)$ when $n$ is of the form $p^4$, $p^2q$, and $pqr$ where $p,q,r$ are distinct prime numbers. These zero divisor graphs are all cochordal. Equivalently, by the result of Fr\"oberg \cite{F}, their edge ideals have a linear free resolution. Motivated by this result, we first classify all $n$ for which the zero divisor graph of $\ZZ/n\ZZ$ is cochordal.

\begin{thm}\label{thm_cochordal_zero_div} Let $n \ge 2$ be a positive integer. Then $\Gamma(\ZZ_n)$ is cochordal if and only if $n$ is one of the following forms:
\begin{enumerate}
    \item $n = p^a$,
    \item $n = p^a q$,
    \item $n = p qr$,
\end{enumerate}
where $p,q,r$ are distinct prime numbers and $a$ is a positive integer.
\end{thm}

We then associate a new invariant to each cochordal graph, called the type sequence. Using the type sequence of a cochordal graph, we compute all the graded Betti numbers of its edge ideal. Consequently, we derive formulae for the Betti numbers of the edge ideals of zero divisor graphs of $\ZZ/n\ZZ$ when $n$ is of the form $n = p^a, n = p^a q$, or $n = pqr$ where $p,q,r$ are distinct prime numbers. Let us now introduce these concepts in more detail.

Beck \cite{B} introduced the zero divisor graph $\Gamma (R)$ of a commutative ring $R$ and studied the finitenistic of the colorings of $\Gamma(R)$ and its algebraic consequences. The structure of zero divisor graphs was furthered analyzed and developed by Anderson and Livingston \cite{AL} and Mulay \cite{M} among others. They particularly focused on a smaller induced subgraph defined by the equivalence classes of zero divisors of $R$, known as the compressed zero divisor graph of $R$. Spiroff and Wickham \cite{SW} carried this study further to investigate the associated primes of $R$. We will build on this idea  further in our study.

\begin{defn}
    Let $R$ be a commutative ring with unity. Denote by $Z(R)$ the set of zero divisors of $R$. The zero divisor graph of $R$, denoted by $\Gamma(R)$ is a graph on the vertex set $Z(R) \backslash \{0\}$ and $\{u,v\}$ is an edge of $\Gamma(R)$ if $uv = 0$.
\end{defn}

Arunkumara, Cameron, Kavaskar, and Tamizh Chelvam \cite{ACKT} recently proved that zero divisor graphs are universal, even when $R$ is restricted to boolean rings, or the rings of integers modulo $n$. In other words, for any finite graph $G$, there exists a positive integer $n$ such that $G$ is an induced subgraph of $\Gamma(\ZZ_n)$. As a first step toward understand the homological invariants of edge ideals of $\Gamma(\ZZ_n)$, we classify all $n$ for which $\Gamma_n$ is cochordal. 

\begin{defn}
    A simple graph $G$ is called a chordal graph if every cycle of length at least $4$ in $G$ has a chord. $G$ is cochordal if its complement $\overline{G}$ is chordal.
\end{defn}
Let $G$ be a simple graph on a vertex set $V(G)$ and edge set $E(G)$. A subset $U \subseteq V(G)$ is called a vertex cover of $G$ if for every edge $\{u,v\} \in E(G)$ either $u$ or $v$ belongs to $U$. We denote by $K_{U,V}$ the complete bipartite graph with a bipartition $V(G) = U \cup V$. When $U = \{u\}$ we also denote $K_{U,V}$ by $K_{u,V}$. Motivated by the Dirac's theorem \cite{Di} on chordal graphs, we define
\begin{defn} Let $\P = (u_k,U_k, u_{k-1}, U_{k-1}, \ldots, u_1, U_1)$ be an ordered set of vertices $u_i$ and subsets $U_i \subseteq V$ of a set of vertices $V$ such that $u_i \notin U_j$ for $j \le i$. We define recursively the graphs $G_i$ as follows
\begin{enumerate}
    \item $G_1 = K_{u_1,U_1}$;
    \item $V(G_j) = V(G_{j-1}) \cup V(K_{u_j,U_j})$, $E(G_j) = E(G_{j-1}) \cup E(K_{u_j,U_j})$.
\end{enumerate} 
We call $\P$ a cochordal constructible system of a graph $G$ if $G = G_k$ and $U_j$ is a vertex cover of $G_{j-1}$ for all $j = 2, \ldots, k$. The type of $\P$ is defined by $\type(P) = (a_k, \ldots, a_1)$ where $a_j = |U_j|$ for all $j = 1, \ldots, k$.
\end{defn}

\begin{thm}\label{thm_cochordal_constructible} A graph $G$ is cochordal if and only if it has a cochordal constructible system $\P$. 
\end{thm}
\begin{defn}
    The type of a cochordal constructible system $\P$ of a cochordal graph $G$ is also called a type of $G$.
\end{defn}
The notion of type of a cochordal graph is motivated from the work of Corso and Nagel \cite{CN1, CN2}. For example, the Ferrers graph $G_\lambda$ associated with a partition $\lambda$ has type $\lambda$. The following is a generalization of the results of Corso and Nagel to arbitrary cochordal graphs.

\begin{thm}\label{thm_betti} The Betti numbers of the edge ideal of a cochordal graph of type $(a_k, \ldots, a_1)$ is given by 
$$\beta_i(S/I) = \binom{a_k}{i} + \binom{a_{k-1}+1}{i} + \cdots + \binom{a_1 + k-1}{i} - \binom{k}{i+1}$$
for all $i \ge 1$.
\end{thm}

We then apply to compute all the Betti numbers of $\Gamma(\ZZ_n)$ when $n$ is of the form $n = p^a$, $n = p^aq$, or $n = pqr$.
\begin{thm}\label{thm_betti_div_1} Let $a \ge 2$ be a positive integer and $p$ be a prime number. Then $I = I(\Gamma(\ZZ_{p^a}))$ has a linear free resolution and  
$$\beta_i(S/I) = \sum_{j= \lceil \frac{a}{2} \rceil}^{a-1} p^{a-j-1}(p-1) \binom{p^j - 2}{i},$$
for all $i \ge 1$. In particular, $\pd(S/I) = p^{a-1} -2$.    
\end{thm}

\begin{thm}\label{thm_betti_div_2} Let $a$ be a positive integer and $p,q$ be two distinct prime numbers. Then $I = I(\Gamma (\ZZ_{p^a q}))$ has a linear free resolution and 
$$ \beta_i(S/I) = \sum_{j= \lceil \frac{a}{2} \rceil}^{a-1} p^{a-j-1}(p-1) \binom{qp^j - 2}{i} + \sum_{j=0}^{\lceil \frac{a}{2} \rceil - 1} p^{a-j-1}(p-1) \binom{(q-1)p^j + p^{a-j} -2}{i},$$
for all $i \ge 1$. In particular, 
$$\pd(S/I) = \begin{cases} 
qp^{a-1} - 2 & \text{ if } a \ge 2, \\
q +p  - 3 & \text{ if } a = 1.\end{cases}$$        
\end{thm}

\begin{thm}\label{thm_betti_div_3} Let $p<q<r$ be prime numbers. Then $I = I(\Gamma(\ZZ_{pqr}))$ has a linear free resolution and 
$$\beta_i(S/I) = (p-1) \binom{qr +p -3}{i} + (q-1) \binom{pr + q -3}{i} + (r-1) \binom{pq+r-3}{i},$$
for all $i \ge 1$. In particular, $\pd(S/I) = qr + p - 3$     
\end{thm}

In Section \ref{sec_cochodal_constructible}, we discuss cochordal constructible systems and prove Theorem \ref{thm_cochordal_constructible}. We then prove Theorem \ref{thm_betti} and deduce its consequences. In Section \ref{sec_zero_div}, we establish Theorem \ref{thm_cochordal_zero_div} and compute all the Betti numbers of edge ideals of zero divisor graphs of $\ZZ/n\ZZ$ when they are cochordal.

\section{Cochordal constructible systems}\label{sec_cochodal_constructible}
In this section, we introduce the notion of cochordal constructible system and prove that any cochordal graph can be constructed from a cochordal constructible system. From that, we deduce a formula for all the Betti numbers of edge ideals of cochordal graphs.

We fix the following notation throughout the paper. Assume that $S = \k[x_1,\ldots,x_n]$ is a standard graded polynomial ring over a field $\k$ with the graded maximal ideal $\m = (x_1,\ldots,x_n)$. An ideal $P$ generated by variables of $S$ is called a monomial prime ideal. If $P = (x_{i_1}, \ldots, x_{i_s})$, we denote by $\supp P = \{i_1, \ldots, i_s\}$.

Let $I = (f_1, \ldots, f_t)$ be a monomial ideal of $S$ and $y$ be a variable of $S$. Let $J = (f_j \mid y \text{ does not divide } f_j)$. Then $I$ has a unique decomposition, called the $y$-partition of $I$, $I = J + yL$, where $yL = (f_j \mid y \text{ divides } f_j)$.

\subsection{Betti numbers} Let $M$ be a finitely generated graded $S$-module. For integers $i,j$ with $i \ge 0$, the $i$-th Betti number of $M$ in degree $j$ is 
$$\beta_{i,j}(M) = \dim_\k \Tor_i^S(k,M)_j.$$
The projective dimension of $M$, denoted by $\pd_S(M)$ and regularity of $M$, denoted by $\reg_S(M)$ are defined by 
\begin{align*}
    \pd_S(M) &= \sup \{ i\mid \beta_{i,j}(M) \neq 0 \text{ for some } j\},\\
    \reg_S(M) &= \sup \{ j -i \mid \beta_{i,j}(M) \neq 0\}. 
\end{align*}
The following result is well-known.
\begin{lem}\label{lem_mul_x} Let $x$ be a variable and $I$ a nonzero homogeneous ideal of $S$. Then 
$$\beta_{i,j}(xI) = \beta_{i,j-1}(I) \text{ for all } i \ge 0.$$
\end{lem}

\subsection{Betti splittings}
Betti splittings were introduced by Francisco, Ha, and Van Tuyl \cite{FHV} for monomial ideals. We follows the treatment of Betti splittings in \cite{NV}. Let $P,I,J$ be proper homogeneous ideals of $S$ such that $P = I + J$.
\begin{defn} The decomposition $P = I + J$ is called a Betti splitting if for all $i \ge 0$, the following equality of Betti numbers holds: $\beta_i(P) = \beta_i(I) + \beta_i(J) + \beta_{i-1}(I \cap J)$.    
\end{defn}

\begin{lem}\label{lem_Koszul_variable_split} Let $I$ be a quadratic monomial ideal and $x$ be a variable. Then the $x$-partition of $I$, $I = xP + J$ is a Betti splitting.
\end{lem}
\begin{proof}
    Since $I$ is generated by quadratic monomials, $P$ is generated by variables. The conclusion follows from \cite[Corollary 4.12]{NV}.
\end{proof}

\subsection{Cochordal constructible systems} Let $G$ be a simple graph on vertex set $V(G) = [n]$ and edge set $E(G) \subseteq V(G) \times V(G)$. The edge ideal of $G$ is defined by 
$$I(G) = (x_ix_j \mid \{i,j\} \text{ is an edge of } G) \subset S.$$

The following is an essential property of cochordal graphs. 
\begin{lem}\label{lem_cochordal_cover_ideal} Let $I = I(G)$ be the  edge ideal of a cochordal graph. Then there exists a variable $x_i$ such that $I:x_i = P$ is a monomial prime ideal. In other words, $I = x_i P + J$ and $J \subset P$.  
\end{lem}
\begin{proof}
    The conclusion follows from \cite[Proposition 3.19]{JV}.
\end{proof}

\begin{proof}[Proof of Theorem \ref{thm_cochordal_constructible}] First, assume that $G$ is cochordal. We prove by induction on $n$ that $G$ has a cochordal constructible system. The base case $n = 2$ is clear. Now, assume that $n \ge 3$. By Lemma \ref{lem_cochordal_cover_ideal}, there exists a variable $x_i$ such that $I(G):x_i = P$ is a monomial prime ideal. Hence, $U = \supp P$ is a vertex cover of $G'$, the induced subgraph of $G$ on $V(G) \setminus \{i\}$. By induction, $G'$ has a cochordal constructible system $\P' = (u_{k-1},U_{k-1}, \ldots, u_1, U_1)$. Set $u_k = i$ and $P_k = U$, we deduce that $\P = (u_k,U_k, \ldots, u_1, U_1)$ is a cochordal constructible system of $G$.

Now, assume that $G$ has a cochordal constructible system $\P$. We prove by induction on $k$ that $G$ is cochordal. By the result of Fr\"oberg \cite{F}, it suffices to prove that $\reg (I(G)) = 2$. By definition, $G = G_k$. The base case $k = 1$ is clear. By induction $G_{k-1}$ is cochordal and $I(G_k) = x_k P_k + I(G_{k-1})$ where $x_k$ corresponds to the vertex $u_k$ and $P_k$ is the monomial prime ideals generated by variables corresponding to $U_k$. By Lemma \ref{lem_Koszul_variable_split}, the decomposition $I(G_k) = x_k P_k + I(G_{k-1})$ is a Betti splitting. Furthermore, by assumption, $x_kP_k \cap I(G_{k-1}) = x_k I(G_{k-1})$. Hence, 
$$\reg (I(G_k))  = \max \{ \reg (I(G_{k-1})), \reg (x P_k), \reg (x_k I(G_{k-1})) - 1\} = 2.$$
The conclusion follows.
\end{proof}

\subsection{Betti numbers of cochordal graphs} Let $G$ be a cochordal graph of type $(a_k, \ldots, a_1)$. By Theorem \ref{thm_cochordal_constructible}, the edge ideal of $G$ can be written as 
$$I(G) = x_kP_k + x_{k-1}P_{k-1} + \cdots +x_1P_1$$
where $P_j$ is a monomial prime ideal generated by $a_j$ variables. Let $I_j = x_{j} P_j + \cdots + x_1 P_1$ for all $j = 1,\ldots, k$.

\begin{proof}[Proof of Theorem \ref{thm_betti}]
    We prove by induction on $k$. The case $k = 1$ is clear as $I_1 = x_1P_1$ with $P_1$ is a monomial prime ideal generated by $a_1$ variables. 

Now, assume that the statement holds for $k-1$. Note that, $\beta_1(S/I) = a_k + \cdots +a_1 = a_k + (a_{k-1} + 1) + \cdots + (a_1 + k-1) - \binom{k}{2}.$ Thus, we may assume that $i \ge 2$. By the proof of Theorem \ref{thm_cochordal_constructible}, the decomposition $I_k = x_kP_k + I_{k-1}$ is a Betti splitting and $x_k P_k \cap I_{k-1} = x_kI_{k-1}$. Hence, we have for all $i \ge 2$,
    $$\beta_i(S/I_k) = \beta_i(S/I_{k-1}) + \beta_{i-1}(S/I_{k-1}) + \beta_i(S/(x_kP_k)).$$
    By induction, we have 
\begin{align*}
    \beta_i(S/I_{k-1}) &= \binom{a_{k-1}}{i} +  \cdots + \binom{a_1 + k-2}{i}  - \binom{k-1}{i+1} \\
    \beta_{i-1} (S/I_{k-1}) &= \binom{a_{k-1}}{i-1} + \cdots +  \binom{a_1 + k-2}{i-1} -\binom{k-1}{i}.
\end{align*}    
The conclusion follows from Lemma \ref{lem_mul_x}, the fact that $\beta_i(S/P_k) = \binom{a_k}{i}$, and the binomial identity $\binom{n}{i} + \binom{n}{i-1} = \binom{n+1}{i}$.
\end{proof}

Dochtermann \cite{Doc} gave a formula for the Betti numbers of the ideal of the complement of a $d$-chordal cluster. His formula requires the computation of the maximal clique sizes containing the exposed edge at each step. Generally, our method is computationally simpler, as it depends on the number of variables, whereas Dochtermann's formula depends on the number of generators of $I$.    

\begin{exm} Consider the following edge ideal 
$$I = x_3 (x_1,x_2) + x_2(x_1,x_4) + x_1(x_5,x_6)\subseteq R = k[x_1,\ldots,x_6].$$
According to Dochtermann's formula, we see that the sequence of exposed edges are $x_1x_6,x_1x_5,x_1x_2,x_1x_3,x_2x_3,x_2x_4$. Hence, the $k_j$-invariants are $0,1,2,3,1,2$. In our formula, $I$ has type $(2,2,2)$. Also, $I$ can be rewritten as
$$I = x_1(x_2,x_3,x_5,x_6) + x_2(x_3,x_4).$$
Thus, $I$ also has type $(4,2)$. By Theorem \ref{thm_betti}, the Betti table of $I$ is 

    \[\begin{array}{l|c c c c }
            &   0   & 1     & 2     & 3 \\ \hline
        -   &   -   &   -   & -     & - \\
      2   &6    &  9 & 5 & 1
    \end{array}
\]
\end{exm}

\begin{cor}\label{cor_pd} Let $G$ be a cochordal graph of type $(a_k, \ldots, a_1)$. Then 
    $\pd (S/I(G)) = \max \{a_k, a_{k-1}+1, \ldots, a_1 + k -1\}$.
\end{cor}
\begin{proof}
    The conclusion follows from Theorem \ref{thm_betti}.
\end{proof}

\begin{rem} Antonino Ficarra pointed out that Lemma 2.4 holds more generally for quadratic monomial ideals having linear free resolutions \cite[Lemma 2.4]{Fi}. Hence, we can extend our concept of type sequence to any quadratic monomial ideal with a linear free resolution. This fact also follows from taking the polarization. In particular, Theorem \ref{thm_betti} and Corollary \ref{cor_pd} hold for arbitrary quadratic monomial ideals having a linear free resolution.    
\end{rem}
\section{Cochordal zero divisor graphs}\label{sec_zero_div} This section is devoted to proving Theorem \ref{thm_cochordal_zero_div} and computing Betti numbers of edge ideals of zero divisor graphs of $\ZZ/n \ZZ$ when $n$ is of the form $n = p^a, n = p^a q$, or $n = pqr$ where $p,q,r$ are distinct prime numbers and $a$ is a positive integer. We first recall the definition of compressed zero divisor graphs.

Let $R$ be a commutative ring with unity. We denote by $Z^*(R)$ the set of nonzero zero divisors of $R$. For $x \in R$, we denote by $[x]$ the set of equivalence class of all $y$ in $R$ such that $\ann(x) = \ann(y)$. 

\begin{defn} The compressed zero divisor graph of $R$, denoted by $\Gamma_E(R)$, is a graph on the vertex set consisting of equivalence classes of elements in $Z^*(R)$ and two distinct equivalence classes $[x]$ and $[y]$ are joined by an edge if and only if $[x] \cdot [y] = [0]$.    
\end{defn}

We now prove that when $n = p^a, n = p^a q$, or $n = pqr$, $\Gamma(R)$ is cochordal and compute its type sequence. Each element of $R = \ZZ / n\ZZ$ has a representative $i$ such that $0 \le i \le n- 1$. Thus, we can view the vertex set of $\Gamma(R)$ as a subset of $\{1,\ldots, n-1\}$. We order the vertices of $\Gamma(R)$ within the same equivalence class by the natural order on $\ZZ$. We also use the following convention: if $\s_1, \ldots, \s_k$ are sequences, then $\s = (\s_k, \ldots, \s_1)$ denotes the concatenation of the sequences $\s_k, \ldots, \s_1$.

\begin{lem}\label{lem_cochord_zero_div_1} Let $p$ be a prime number and $a$ be a positive integer at least $2$. Let 
$$\s_i = (p^i - p^{a-i-1}-1, p^i - p^{a-i-1} -2, \ldots, p^i - p^{a-i}) \text{ for } i=\lceil \frac{a}{2} \rceil, \ldots,a-1.$$ 
Then $\Gamma(\ZZ_{p^a})$ is cochordal and has a type $\s = (\s_{a-1}, \ldots, \s_{\lceil \frac{a}{2} \rceil})$.
\end{lem}
\begin{proof} Let $n = p^a$ and $R = \ZZ/n \ZZ$. For each $i = 1, \ldots, a-1$, we denote by 
$$V_i = \{x \in R \mid [x] = [p^i]\}.$$ 
First, we have $|V_i| = \varphi(p^{a-i}) = \varphi_i = p^{a-i-1}(p-1)$, where $\varphi(n)$ is the Euler's function. We assume that $V_i = \{v_{i,1}, \ldots, v_{i,\varphi_i}\}$ and $v_{i,j} < v_{i,j+1}$ for all $i$ and $j$. We now define 
    \begin{equation*}
        U_{i,j} = V_{i-1} \cup V_{i-2} \cup \cdots \cup V_{a -i} \cup \{ v_{i,1}, \ldots, v_{i,j-1} \}
    \end{equation*}
    for $i = \lceil \frac{a}{2} \rceil, \ldots, a-1$ and $j = 1, \ldots, \varphi_i$. Note that $\{v_{i_1,j_1},v_{i_2,j_2}\}$ is an edge of $\Gamma_n$ if and only if $i_1 + i_2 \ge a$. Hence, the ordered set $\P = (v_{i,j},U_{i,j} \mid i = a-1, \ldots, \lceil \frac{a}{2} \rceil, j = \varphi_i, \ldots, 1)$ is a cochordal constructible system of $\Gamma(\ZZ_n)$. By Theorem \ref{thm_cochordal_constructible}, $\Gamma(\ZZ_n)$ is cochordal. Finally, we have 
    \begin{align*}
        |U_{i,j}| & = |V_{i-1}| + \cdots + |V_{a-i}| + j-1\\
                   & = p^{a-i}(p-1) \left ( 1  + \cdots + p^{2i-a-1} \right ) + j - 1 \\
                   & = p^i - p^{a-i}  + j -1,
    \end{align*}
    for all $i = \lceil \frac{a}{2} \rceil, \ldots, a-1$ and $j = 1, \ldots, \varphi_i$. The conclusion follows.
\end{proof}

\begin{lem}\label{lem_cochord_zero_div_2} Let $p,q$ be distinct prime numbers and $a$ be a positive integer. Let 
\begin{align*}
    \s_i &= (qp^i - p^{a-i-1}-1, qp^i - p^{a-i-1} -2, \ldots, qp^i - p^{a-i}) \text{ for }i=\lceil \frac{a}{2} \rceil,\ldots, a-1,\\
    \s_i & = ((q-1)p^i + p^{a-i-1}(p-1) -1, (q-1)p^i + p^{a-i-1}(p-1) -2, \ldots, (q-1)p^i) \\
        & \text{ for } i = 0, \ldots, \lceil \frac{a}{2} \rceil - 1.
\end{align*} 
Then $\Gamma(\ZZ_{p^aq})$ is cochordal and has a type $\s = (\s_{a-1}, \ldots, \s_0)$.
\end{lem}
\begin{proof} Let $n = p^aq$ and $R = \ZZ/n \ZZ$. We set
\begin{align*}
    U_i &= \{x \in R \mid [x] = [p^i]\} \text{ for } i = 1, \ldots, a,\\
    V_i &= \{x \in R \mid [x] = [p^iq]\} \text{ for } i = 0, \ldots, a-1.
\end{align*}
Then, $|U_a| = q-1$ and 
\begin{align*}
    |U_i| & = \varphi(p^{a-i}q) = p^{a-i-1}(p-1)(q-1) \text{ for } i =1, \ldots, a-1\\
    |V_i| & = \varphi(p^{a-i}) = \varphi_i = p^{a-i-1}(p-1) \text{ for } i = 0, \ldots, a-1. 
\end{align*}
We assume that $V_i = \{v_{i,1}, \ldots, v_{i,\varphi_i}\}$ and $v_{i,j} < v_{i,j+1}$ for all $i$ and $j$. We now define 
    \begin{equation*}
        W_{i,j} = U_{a} \cup U_{a-1} \cup \cdots \cup U_{a -i} \cup V_{i-1} \cup \cdots \cup V_{a-i} \cup  \{ v_{i,1}, \ldots, v_{i,j-1} \}
    \end{equation*}
    for $i = 0, \ldots, a-1$ and $j = 1, \ldots, \varphi_i$. For $x\in U_i$ and $y \in V_j$, we have $\{x,y\}$ is an edge of $\Gamma_n$ if and only if $i + j \ge a$. For $x \in V_i$ and $y \in V_j$, we have $\{x,y\}$ i an edge of $\Gamma_n$ if and only if $i + j \ge a$. Hence, the ordered set $\P = (v_{i,j},W_{i,j} \mid i = a-1, \ldots, 0, j = \varphi_i, \ldots, 1)$ is a cochordal constructible system of $\Gamma(\ZZ_n)$. By Theorem \ref{thm_cochordal_constructible}, $\Gamma(\ZZ_n)$ is cochordal. By the proof of Lemma \ref{lem_cochord_zero_div_1}, we have 
    \begin{align*}
        |W_{i,j}| & = |U_{a}| + \cdots + |U_{a-i}| + |V_{i-1}| + \cdots + |V_{a-i}| + j-1\\
                   & = (q-1) + (q-1)(p-1) ( 1 + \cdots + p^{i-1}) + p^i - p^{a-i} + j - 1 \\
                   & = (q-1)p^i + p^i - p^{a-i}  + j -1\\
                   & = qp^i - p^{a-i} + j -1,
    \end{align*}
    for all $i = \lceil \frac{a}{2} \rceil, \ldots, a-1$ and $j = 1, \ldots, \varphi_i$. For $i = 0, \ldots, \lceil \frac{a}{2} \rceil - 1$ and $j = 1, \ldots, \varphi_i$, we have 
    \begin{align*}
        |W_{i,j}| & = |U_{a}| + \cdots + |U_{a-i}| + j-1\\
                   & = (q-1)p^i + j -1.
    \end{align*}
    The conclusion follows.
\end{proof}

\begin{lem}\label{lem_cochord_zero_div_3} Let $p,q,r$ be distinct prime numbers. Let 
\begin{align*}
    \s_1 & = (r + (p-1)(q-1) -2, \ldots, (p-1)(q-1)) \\
    \s_2 & = (pr - p +q - 2, \ldots, pr - p) \\
    \s_3 & = (qr +p - 3, \ldots, qr - 1).
\end{align*}
Then $\Gamma(\ZZ_{pqr})$ is cochordal and has a type $\s = (\s_3, \s_2, \s_1).$
\end{lem}
\begin{proof} Let $n = pqr$ and $R = \ZZ/n \ZZ$. We assume that $p < q <r$ and set 
\begin{align*}
U_1 &= \{x \in R\mid [x] = [p], U_2 = \{x \in R \mid [x] = [q] \}, U_3 = \{x \in R \mid [x] = [r]\}\\
V_1 &= \{x\in R \mid [x] = [pq]\}, V_2 = \{x \in R \mid [x] = [pr]\}, V_3 = \{x \in R \mid [x] = [qr]\}.
\end{align*}
Let $\varphi_1 = |V_1| = r-1$, $\varphi_2 = |V_2| = q-1$, and $\varphi_3 = |V_3| = (p-1)$. Write $V_i = \{v_{i,1}, \ldots, v_{i,\varphi_i}\}$ and $v_{i,j} < v_{i,j+1}$ for all $i=1, 2,3$ and $j = 1, \ldots, \varphi_i$. We now define 
    \begin{align*}
        W_{1,j} & = U_{3} \cup  \{ v_{1,1}, \ldots, v_{1,j-1} \} \text{ for } j = 1, \ldots, \varphi_1\\
        W_{2,j} & = V_1 \cup U_2 \cup \{ v_{2,1},\ldots,v_{2,j-1} \} \text{ for } j=1, \ldots, \varphi_2 \\
        W_{3,j} & = V_1 \cup V_2 \cup U_1 \cup \{v_{3,1}, \ldots, v_{3,j-1} \} \text{ for } j = 1, \ldots, \varphi_3. 
    \end{align*}
    Then the ordered set $\P = (v_{i,j},W_{i,j} \mid i = 3, 2, 1, j = \varphi_i, \ldots, 1)$ is a cochordal constructible system of $\Gamma(\ZZ_n)$. By Theorem \ref{thm_cochordal_constructible}, $\Gamma(\ZZ_n)$ is cochordal. Finally, we have 
    \begin{align*}
        |W_{1,j}| & = |U_{3}| + j-1 = (p-1)(q-1) + j -1 \text{ for } j = 1, \ldots, \varphi_1,\\
        |W_{2,j}| & = |V_1| + |U_2| + j - 1 = pr -p +j-1 \text{ for } j = 1, \ldots, \varphi_2,\\
        |W_{3,j}| & = |V_1| + |V_2| + |U_1|  + j -1 = qr + j -2 \text{ for } j = 1, \ldots, \varphi_3.
    \end{align*}
    The conclusion follows.
\end{proof}
We are now ready for 
\begin{proof}[Proof of Theorem \ref{thm_cochordal_zero_div}] First, we assume that $G = \Gamma(\ZZ_n)$ is cochordal. In particular, its induced matching number is $1$. We divide the proof into several steps.

\smallskip

\noindent \textbf{Step 1.} $n$ has at most $3$ prime factors. Assume by contradiction that $n$ has more than $3$ prime factors, say $n = p_1^{a_1}p_2^{a_2} p_3^{a_3} p_4^{a_4} m$ where $m$ is a positive integer and $p_1,\ldots,p_4$ are distinct prime numbers not dividing $m$. Let $x = p_1^{a_1}p_2^{a_2} m$, $y = p_3^{a_3}p_4^{a_4} m$, $u = p_1^{a_1} p_3^{a_3} m$ and $v = p_2^{a_2} p_4^{a_4} m$. Then $\{x,y\}$ and $\{u,v\}$ form an induced matching of $G$, which is a contradiction.

\smallskip

\noindent \textbf{Step 2.} If $n$ has $3$ prime factors then $n = pqr$. Assume by contradiction that $n = p^a q^b r^c$ for some $a > 1$ and $b,c$ are positive integers. Let $x = p^a$, $y = q^br^c$, $u = p^{a-1} q^b$, and $v = p r^c$. Then $\{x,y\}$ and $\{u,v\}$ form an induced matching of $G$, a contradiction.

\smallskip

\noindent \textbf{Step 3.} If $n$ has $2$ prime factors then $n = p^a q$. Assume by contradiction that $n = p^a q^b$ where $a,b > 1$ are positive integers. Let $x = p^a$, $y = q^b$, $u = p^{a-1} q^{b-1}$ and $v = pq w$ where $w$ is any prime divisor of $pq-1$. Then $u \neq v$ and $\{x,y\}$ and $\{u,v\}$ form an induced matching of $G$, a contradiction.  

The conclusion then follows from Lemma \ref{lem_cochord_zero_div_1}, Lemma \ref{lem_cochord_zero_div_2}, and Lemma \ref{lem_cochord_zero_div_3}.
\end{proof}

We now apply Theorem \ref{thm_betti} to compute all the Betti numbers of $\Gamma(\ZZ_n)$ when $n$ is of the form $n = p^a$, $n = p^aq$, or $n = pqr$.

\begin{proof}[Proof of Theorem \ref{thm_betti_div_1}]
    Let $\s_j = (s_{j,1}, \ldots, s_{j,\varphi_j})$ with $\varphi_j = p^{a-j-1}(p-1)$ and     \begin{align*}
        s_{j,1} & = p^j - p^{a-j-1} - 1 \\
        s_{j,\ell} & = s_{j,\ell-1} - 1 \text{ for } \ell = 2, \ldots, \varphi_j.
    \end{align*} 
    By Lemma \ref{lem_cochord_zero_div_1}, $\Gamma(\ZZ_n)$ has a type $\s = (\s_{a-1}, \ldots, \s_{\lceil \frac{a}{2} \rceil })$. By Theorem \ref{thm_betti}, we deduce that 
    $$\beta_i(S/I) = \sum_{j= \lceil \frac{a}{2} \rceil}^{a-1} \varphi_j \binom{s_{j,1} + \sum_{\ell = j+1}^{a-1} \varphi_\ell}{i}.$$
    For $j < a-1$, we have 
    \begin{equation}\label{eq_sum}
        \sum_{\ell = j+1}^{a-1} \varphi_\ell = (p-1) \sum_{\ell = j+1}^{a-1} p^{a-\ell - 1} = p^{a-j-1} - 1.
    \end{equation}
    The conclusion follows.
\end{proof}

\begin{proof}[Proof of Theorem \ref{thm_betti_div_2}] Let $\s_j = (s_{j,1}, \ldots, s_{j,\varphi_j})$ with $\varphi_j = p^{a-j-1}(p-1)$ be such that $s_{j,\ell} = s_{j,\ell-1} - 1$ for $\ell = 2, \ldots, \varphi_j$ and 
\begin{align*}
s_{j,1} &= qp^j - p^{a-j-1} - 1 \text{ for } j = \lceil \frac{a}{2} \rceil, \ldots, a-1\\
s_{j,1} & = (q-1)p^j + p^{a-j-1}(p-1) - 1 \text{ for } j = 0, \ldots, \lceil \frac{a}{2} \rceil - 1.
\end{align*}
By Lemma \ref{lem_cochord_zero_div_2}, Theorem \ref{thm_betti}, and the proof of Theorem \ref{thm_betti_div_1}, we deduce that
    $$\beta_i(S/I) = \sum_{j=0}^{a-1} \varphi_j \binom{s_{j,1} + \sum_{\ell = j+1}^{a-1} \varphi_\ell}{i}.$$
    The conclusion then follows from Eq. \eqref{eq_sum}.
\end{proof}

\begin{proof}[Proof of Theorem \ref{thm_betti_div_3}]
    The conclusion follows from Lemma \ref{lem_cochord_zero_div_3} and Theorem \ref{thm_betti}.
\end{proof}

\begin{rem} Rather, Imran, and Pirzada \cite{RIP} and Pirzada and Rather \cite{PR} used Hochster's formula \cite[Theorem 8.1.1]{HH} to compute the Betti numbers of the edge ideals of $\Gamma_n$ when $n$ is of the form $p^a, p^{a}q$, and $pqr$. The authors provided closed formulae when $a \le 4$ in the first case and $a \le 2$ in the second case. Our method is entirely different and yields a complete answer. Additionally, our formulae are simpler. By Theorem \ref{thm_betti_div_3}, the last Betti number of $\Gamma_{pqr}$ is $p-1$. Hence, the the first statement in \cite[Proposition 4.8]{RIP} is incorrect.    
\end{rem}

\begin{rem} Theorem \ref{thm_betti} is useful for calculating the Betti numbers of other classes of cochordal graphs. We will explore this further in subsequent work.    
\end{rem}
\section*{Acknowledgments} Part of this work was developed by Professor Nguyen Cong Minh and the second author in a separate project. We extend our gratitude to Professor Nguyen Cong Minh for kindly allowing us to include some material in this project and for his valuable suggestions. Part of this work was completed during the first author’s visit to the Vietnam Institute for Advanced Study in Mathematics (VIASM). He wishes to express his gratitude to VIASM for its hospitality.

\end{document}